\documentclass[12pt, A4]{article}
\usepackage{amsmath,amssymb,amsfonts,mathrsfs,hyperref,microtype, amsthm}

\usepackage{geometry}
\usepackage{authblk}

\usepackage[pdftex]{graphicx}
\usepackage{eufrak}
\usepackage{graphicx,color}
\usepackage[T1]{fontenc}
\usepackage{rotating}
\usepackage{booktabs}
\usepackage{mathtools}
\usepackage{float}
\usepackage{breqn}
\usepackage{graphicx}
\usepackage{caption}
\usepackage[pdftex]{graphicx}
\usepackage{pgfplots}
\usepackage{tikz}

\newtheorem{thm}{Theorem}[section]
\newtheorem{rem}{Remark}[section]
\newtheorem{definition}{Definition}[section]
\newtheorem{lem}{Lemma}[section]
\newtheorem{prop}{Proposition}[section]
\newtheorem{cor}{Corollary}[section]

\numberwithin{equation}{section}

\def\HH{ \EuFrak H}
\def\N{{\rm I\kern-0.16em N}}
\def\R{{\rm I\kern-0.16em R}}
\def \E{{\rm I\kern-0.16em E}}
\def\P{{\rm I\kern-0.16em P}}
\def\F{{\rm I\kern-0.16em F}}
\def\B{{\rm I\kern-0.16em B}}
\def\C{{\rm I\kern-0.46em C}}
\def\G{{\rm I\kern-0.50em G}}

\newcommand{\PP}{\mathscr{P}}

\newcommand{\LG}{\mathrm{L}}
\newcommand{\ud}{\mathrm{d}}
\newcommand{\LL}{\mathfrak{L}}
\newcommand{\RR}{\mathfrak{R}}
\numberwithin{equation}{section}

\font\eka=cmex10
\usepackage{ae}

\def\ind{\mathrel{\hbox{\rlap{%
\hbox to 7.5pt{\hrulefill}}\raise6.6pt\hbox{\eka\char'167}}}}
\begin{document}

\title{\Large{\bf On a new Sheffer class of polynomials related to normal product distribution }}

\author[1]{Ehsan Azmoodeh}
\author[2]{Dario Gasbarra}
\affil[1]{Ruhr University Bochum}
\affil[2]{University of Helsinki}
\date{ }                     
\setcounter{Maxaffil}{0}
\renewcommand\Affilfont{\itshape\small}

\maketitle

\abstract  Consider a generic random element $F_\infty= \sum_{\text{finite}} \lambda_k (N^2_k -1)$ in the second Wiener chaos with a finite number of non-zero coefficients in the spectral representation where $(N_k)_{k \ge 1}$ is a sequence of i.i.d $\mathscr{N}(0,1)$. Using the recently discovered (see  Arras  et al. \cite{a-a-p-s-stein}) stein operator $\RR_\infty$ associated to $F_\infty$,  we introduce a new class of polynomials $$\PP_\infty:= \{ P_n = \RR^n_\infty \textbf{1}  \, : \, n \ge 1 \}.$$ We analysis in details the case where $F_\infty$ is distributed as the normal product distribution $N_1 \times N_2$, and relate the associated polynomials class to Rota's {\it Umbral calculus} by showing that it is a \textit{Sheffer family} and enjoys many interesting properties. Lastly, we study the connection between the polynomial class $\PP_\infty$ and the non-central probabilistic limit theorems within the second Wiener chaos. 

 \vskip0.3cm
\noindent {\bf Keywords}:
Second Wiener chaos, Normal product distribution, Cumulants/Moments, Weak convergence, Malliavin Calculus, Sheffer polynomials, Umbral Calculus

\noindent{\bf MSC 2010}: 60F05, 60G50, 46L54, 60H07, 26C10
 \begin{small}
 \tableofcontents
 \end{small}

\section{Introduction}

The motivation of our study comes from the subsequent facts on the Gaussian distribution. Let $N \sim \mathscr{N}(0,1)$ be a standard Gaussian random variable. Consider the following well known first order differential operator related to the so called {\it Ornstein--Uhlenbeck} operator  $$(L f ) (x) =x f(x) - f'(x) = - e^{\frac{x^2}{2}} \frac{d}{dx} \big(  e^{-\frac{x^2}{2}} f \big) $$ acting on a  suitable class $\mathcal{F}$ of test functions $f$. A fundamental result in realm of Stein method in probabilistic approximations, known as stein charactrization of the Gaussian distribution, reads that for a given random variable $F \sim N$  if and only if $\E \left[ (Lf) (F) \right]= \E[ F f (F ) - f'(F) ]=0$ for $f \in \mathcal{F}$ (in fact, the polynomials class is enough). The second notable feature of the operator $L$ in connection with the Gaussian distribution is the following.  Pual Malliavin in his book \cite[page 231]{M-book}, for every  $n \in \N_0$, define the so called {\it Hermite polynomial} $H_n$ of order $n$ using the relation $H_n (x) := L^n \textbf{1} (x)$. For example, the few first Hermite polynomials are given by $H_0 (x)=1, \, H_1(x)=x, \, H_2 (x) = x^2 -1, \, H_3 (x)= x^3 - 3x, H_4(x)= x^4 - 6x +3$. One of the significant properties of the Hermite polynomials is that they constitute an orthogonal polynomials class with respect to the Gaussian measure $ \frac{e^{- x^2/2}}{\sqrt{2 \pi}} \ud x$. The orthogonality character of the Hermite polynomials can be routinely seen as a direct consequence of the adjoint operator $L^\star = \frac{d}{dx}$ that is straightforward computation.\\

Instead the Gaussian distribution (living in the first Wiener chaos) we consider distributions in the second Wiener chaos having a finite number of non-zero coefficients in the spectral representation, namely random variables of the form 
\begin{equation}\label{eq:all-targets}
F_\infty= \sum_{k=1}^{d} \lambda_k (N^2_k -1), \quad d \ge 1,
\end{equation}
 where $(N_k)_{k \ge 1}$ is a sequence of i.i.d $\mathscr{N}(0,1)$. Relevant examples of such those random elements are {\it centered chi--square} and {\it normal product} distributions corresponding to the cases when all $\lambda_k$ are equal, and $d =2$ where two non--zero coefficients $\lambda_1 = - \lambda_2$ respectively. The target distributions of the form $(\ref{eq:all-targets})$ appears often in the classical framework of limit theorems of $U$-statistics, see \cite[Chapter 5.5 Section 5.5.2]{U-statistic}. Noteworthy,  recently Bai \& Taqqu in \cite{bai} showed that the distributions of the form $(\ref{eq:all-targets})$ with $d=2$ can be realized as the limit in distribution of the fractal stochastic process {\it generalized Rosenblatt process} 
 \begin{align*}
  Z_{\gamma_1,\gamma_2} (t) =\int_{\mathbb{R}^2}^{\prime} \bigg(\int_0^t (s-x_1)^{\gamma_1}_+(s-x_2)^{\gamma_2}_+ds\bigg)B(d x_1)B(d x_2)
 \end{align*}
when the exponents $(\gamma_1,\gamma_2)$ approach the boundary of the triangle $\Delta := \big\{ (\gamma_1,\gamma_2)  \, \vert \, -1 < \gamma_1, \gamma_2 < - 1/2 , \, \gamma_1 + \gamma_2 > - 3/2\big\}$. Here, $B$ stands for the Brownian measure, and the prime $\prime$ indicates the off--diagonal integration.  One of their interesting results reads as  $$ Z_{\gamma_1,\gamma_2} (1)  \stackrel{law}{\longrightarrow} N_1 \times N_2, \quad \text{ as } \quad  (\gamma_1,\gamma_2) \to (-1/2,\gamma), -1 < \gamma < - 1/2.$$

 Recently, the authors of $\cite{a-a-p-s-stein}$, using two different approaches, one based on Malliavin Calculus, and the other relying on Fourier techniques, for probability distributions of the form $(\ref{eq:all-targets})$, introduced the following so called stein differential operator of order 
 $d (=\text{the number of the non--zero coefficients})$
 \begin{equation}\label{eq:SME} 
\RR_\infty f (x):= \sum_{l=2}^{d+1} (b_l - a_{l-1} x ) f^{(d+2-l)}(x) - a_{d+1} x f(x),
\end{equation}
where the coefficients $(a_l)_{1 \le l \le d+1}, (b_l)_{2 \le l \le d+1}$ are akin to the random element $F_\infty$ through the relations; 
\begin{align*}
   a_l= \frac{P^{(l)}(0)}{l! 2^{l-1}}, \quad \text{ and } \quad 
   b_l=     \sum_{r=l}^{d+1} \frac{a_r}{(r-l+1)!} \kappa_{r-l+2}(F_\infty),
   \end{align*}
and $ P(x)= x \prod_{i=1}^{d}(x - \lambda_k) $. Here $\kappa_r (F)$ stands for the $r$th cumulant of the random variable $F$. In this paper,  the case $d=2$ of two non--zero coefficients with particular parametrization $\lambda_1 = - \lambda_2 = \frac{1}{2}$ is of our interest. The operator $\RR_\infty$ then admits the form 
\begin{equation}\label{eq:main-operator}
\RR_\infty f (x) = \RR f (x) := x f(x) - f'(x) -x f''(x).
\end{equation}
Note that in this setup the random variable $F_\infty = N_1 \times N_2$ (equality in distribution) is the normal product distribution. The stein operator $(\ref{eq:main-operator})$ associated to the normal product distribution first introduced by Gaunt in \cite{g-2normal}. The normal product distribution also belongs to a wide class of probability distributions known as the {\it Variance--Gamma} class, consult \cite{g-variance-gamma} for further details and development of the Stein characterisations. Following the Gaussian framework, we define the polynomial class 
\begin{equation}\label{eq:main-poly}
\PP := \{ P_n: =\RR^n \textbf{1}  \, : \,  n \in \N_0 \}
\end{equation}
where operator $\RR$ is the same one as in $(\ref{eq:main-operator})$. 
The first fifteenth polynomials $P_n$ are presented in the Appendix Section \ref{sec:appendixa}. In this short note, we study some properties of the polynomials class $\PP$. We derive, among other results, that the class $\PP$ is a {\it Sheffer family} of polynomials, hence possess a rich structure, and can be analyzed within the Gian-Carlo Rota's {\it Umbral Calculus}.  See Section \ref{sec:sheffer} for definitions. We ends the note with connection of the polynomial class $\PP$ to the non--central probabilisitic limit theorems, and show that polynomial $P_6 \in \PP$ plays a crucial role in limit theorems when the target distribution is the favourite  normal product random variable $N_1 \times N_2$.  

\section{Normal product distribution $N_1\times N_2$}

In this section, we briefly collect some properties of normal product distribution. 
\subsection{Modified Bessel functions of the second kind}

The {\it modified Bessel functions} $I_\nu$ and $K_\nu$ with index $\nu$ of the first and the second kinds respectively are defined as two independent solutions of the so called modified Bessel differential equation 
\begin{equation}\label{eq:mbde}
[\RR^{\nu}_{MB}f] (x) := x^2 f''(x) + x f'(x) - (x^2 +\nu^2)f(x)=0
\end{equation}
 with the convention $\RR^{0}_{MB}= \RR_{MB}$. We collect the following results on modified Bessel function of the second kind and the normal product distribution.  
\begin{itemize}
\item[(i)] It is well known that (see e.g. \cite{product-density}) the density function $p_\infty$ of the normal product random variable is given by $$p_\infty(x)=\frac{1}{\pi} K_0(\vert x \vert) \qquad x \in \R$$ where $K_0$ be the modified Bessel function of the second kind with the index $\nu=0$.  

\item[(ii)] The modified Bessel function of the second kind $K_0$ possess several useful representation. Among those, here we state
$$ K_0(\vert x \vert ) = \frac{1}{2} G^{2,0}_{0,2} \left( \frac{x^2}{4} \big \vert 0,0 \right)$$
where $G$ here is the so called \textsc{Meijer $G$-function} that shares many interesting properties, see e.g. \cite{handbook}.

\item[(iii)] $\frac{d}{dx} K_0 (x) = - K_1 (x)$, $K_0 (x) \sim - \log (x)$ as $x \downarrow 0$, and $K_0(x) \sim \sqrt{\frac{\pi}{2x}} e^{-x}$ as $x \to \infty$. 
\item[(iv)] The relation $-x [\RR f] (x)= x^2 f''(x) + x f(x) - x^2 f(x) =  [\RR_{MB}f] (x)$ holds, where $\RR$ as in $(\ref{eq:main-operator})$ is the stein operator associated to the normal product distribution.  
\item[(v)] The characteristic function of the normal product distribution is $\varphi_\infty (t) = (1 + t^2)^{-1}$. Hence, the normal product distribution is the unique random variable in the second Wiener chaos having only two non--zero coefficients equal to $\lambda_1 = - \lambda_2 = 1/2$. 
\item[(vi)] for $n \in \N$,  $$\mu_{2n} (N_1 \times N_2) := \E[ (N_1 \times N_2)^{2n}] = ((2n-1)!!)^2, \quad \kappa_{2n}( N_1 \times N_2 ) = (2n-1)!$$ and $\mu_{2n-1}(N_1 \times N_2)= \kappa_{2n-1}(N_1 \times N_2) =0$.
\end{itemize}

\subsection{The adjoint operator $\RR^\star$}

We recall the following well known finite dimensional Gaussian integration by parts formulae. 
\begin{lem}\label{lem:fdgibp}
Let $N_1,\dots,N_T$ be i.i.d. $\mathscr{N}(0,1)$. For smooth random variables
\begin{align*}
  F( N_1, \dots, N_T),  \quad \text{ and } \quad  \Big( u_t(N_1,\dots, N_T)\Big)_{ t=1,\dots T}
\end{align*}
we have the finite dimensional Gaussian integration by parts formula
\begin{align*}
  \E\bigl( \langle DF, u \rangle\bigr) = \E\bigl( F  \delta(u) \bigr)
\end{align*}
where $ \delta( u) = \sum_{t=1}^T u_t N_t - \sum_{t=1}^T D_t u_t$ is the Skorokhod integral in the finite dimension.
\end{lem}

\begin{prop}\label{prop:adjoint} 
Let $F_\infty = N_1 \times N_2$, and $\mu_\infty= p_\infty dx$ the associated probability measure on the real line. Consider the second order differential operator $$(\RR f)(x) = x f(x) - f'(x) - x f''(x).$$ Then the adjoint operator $\RR^{\star}$ in the space $L^2(\R,\mu_\infty)$ is 
\begin{align*}
(\RR^* g) (x)= (\RR g) (x) + \theta (x) g'(x)- x g (x),    
\end{align*}
where the special function $\theta$ is given by the conditional expectation
\begin{align*}
\theta(x)= \E( N_1^2+ N_2^2 \big \vert  N_1 N_2=x )= 2 \vert x \vert \frac{K_1 (\vert x \vert)}{K_0(\vert x \vert)}.
\end{align*}
\end{prop}

\begin{proof}
In order to compute the adjoint of $\RR^\star$, with $u =   \left(
\begin{matrix}
u_1\\ u_2\end{matrix} \right)=\frac 1 2 \left(
\begin{matrix}
N_1\\ N_2\end{matrix} \right)$, and $ v= \frac 1 2 \left(
\begin{matrix}
N_2\\ N_1\end{matrix} \right)$ we write
\begin{equation*} 
\begin{split}
\E & \bigl( f''( N_1 N_2) g(N_1 N_2) N_1 N_2 \bigr) = \E\biggl( \bigl\langle D f'(N_1N_2) ,  g(N_1N_2) u\bigr \rangle\biggr) \\
&= \E\biggl( f'( N_1 N_2) \delta\bigl(  g(N_1 N_2 ) u \bigr ) \biggr) \\
&= \E\biggl( f'(N_1N_2) \bigl(   g(N_1 N_2) 
\frac{( N_1^2 + N_2^2) } 2 - g'(N_1 N_2) N_1 N_2  -g(N_1 N_2)      ) \biggr)
\end{split}
\end{equation*}
where
\begin{equation*} 
\begin{split}
 \E\biggl( &f'(N_1N_2) g'(N_1 N_2)  N_1 N_2 \biggr) \\
 & =  \E\biggl( f(N_1N_2) \biggl(   g'(N_1 N_2) \frac{( N_1^2 + N_2^2) } 2 - g''(N_1 N_2) N_1 N_2 -g'(N_1N_2) \biggr)\biggr)
\end{split}
\end{equation*}
and
\begin{equation*} 
\begin{split}
\E\biggl( f'(N_1N_2)  & \bigl(   g(N_1 N_2) \frac{( N_1^2 + N_2^2) } 2\bigr) \biggr)=\E\biggl( \bigl \langle D f(N_1N_2), g(N_1 N_2) v\bigr\rangle \biggr)\\
&= \E\biggl(  f(N_1N_2) \delta\bigl( g(N_1 N_2) v\bigr)\biggr) \\
&=  \E\biggl( f(N_1 N_2) \biggl( g(N_1N_2) N_1N_2 - g'(N_1N_2) 
\frac{ ( N_1^2+ N_2^2 ) } 2 \biggr) \biggr).
\end{split}
\end{equation*}

Therefore we have
\begin{equation*} 
\begin{split}
  \E\bigl(  & \RR f(N_1 N_2) g(N_1 N_2) \bigr) =    \E\biggl(  f(N_1 N_2 ) g(N_1N_2)N_1 N_2\biggr) - \E\biggl(  f'(N_1 N_2 ) g(N_1N_2)\biggr) \\
  & \quad -\E\biggl(  f''(N_1 N_2 ) g(N_1N_2)N_1 N_2\biggr)\\
& = \E\biggl(  f(N_1 N_2 ) g(N_1N_2)N_1 N_2\biggr)-\E\biggl(  f'(N_1 N_2 ) g(N_1N_2)\biggr)\\
& \quad +  \E\biggl( f(N_1N_2) \biggl(   g'(N_1 N_2) \frac{( N_1^2 + N_2^2) } 2 - g''(N_1 N_2) N_1 N_2 -g'(N_1N_2)\biggr)\biggr) \\
& \quad + \E\bigl( f'(N_1N_2)g(N_1N_2) \bigr)- \E\biggl( f(N_1 N_2) \biggl( g(N_1N_2) N_1N_2 - g'(N_1N_2) 
\frac{ ( N_1^2+ N_2^2 ) } 2 \biggr) \biggr) \\
&=  \E\biggl( f(N_1N_2) \biggl(   g'(N_1 N_2) ( N_1^2 + N_2^2) - g''(N_1 N_2) N_1 N_2 -g'(N_1 N_2) \biggr)\biggr)\\
&= \E\biggl(  f( N_1N_2)  \biggl(   g'(N_1 N_2) \E\bigl( N_1^2 + N_2^2 \big\vert N_1N_2 \bigr) - g''(N_1 N_2) N_1 N_2  -g'(N_1N_2) \biggr)\biggr).
\end{split}
\end{equation*}

This implies that $ \RR^* g(x)=\eta(x)g'(x)-g''(x)x  $ where  the conditional expectation
\begin{align*}
\eta(x)= \E( N_1^2+ N_2^2 \big \vert  N_1 N_2=x ) -1
\end{align*}
is  a special function. In order to compute explicitly the special function $\eta$, we make use of Lemma \ref{lem:appendixb} to write 
\begin{align*} &&
 \eta(x)= \E(  N_1^2+ N_2^2 \big \vert  N_1 N_2=x  )  -1 = \frac{  \E( \delta_x( N_1 N_2) ( N_1^2+ N_2^2 )  ) }{ p_\infty(x) } -1
 \end{align*}
where
\begin{align*}&
 \E\bigl(  (N_1^2+ N_2^2) \delta_x( N_1 N_2) 
\bigr)  = \frac  1 {2\pi} \int_{\R} \biggl( \int_{\R}\bigl (z^2 + y^2 \bigr)\delta_x( zy)  e^{-\frac{z^2} 2} dz \biggr)  e^{-\frac{y^2} 2} dy & \\ &
 = \frac  1 {2\pi} \int_{\R} \biggl( \int_{\R} \biggl(\frac{ u^2}{y^2} + y^2 \biggr)\delta_x(u)  e^{-\frac{u^2} {2 y^2} } y^{-1} du \biggr)
 e^{-\frac{y^2} 2} dy & \\ &
 = \frac  1 {2\pi} \int_{\R} \biggl (\frac{ x^2}{y^2} + y^2 \biggr) e^{-\frac{x^2} {2 y^2} } |y|^{-1}  e^{-\frac{y^2} 2} dy=
 \frac  1 {2\pi} \int_0^{\infty}  \biggl (\frac{ x^2}{v} + v  \biggr ) e^{-\frac{x^2} {2v} -\frac{v} 2}  v^{-1} dv
 & \\ & = \frac{ 2 |x| K_1( |x| ) }{\pi}
&\end{align*}
by changing variables first  with $z=u/y$ and then with $v=y^2$. Therefore $ \eta(x) =  \frac{ 2 |x| K_1( |x| ) - K_0(|x|) }{  K_0(|x|)}$, and the result follows at once taking into account the definition of operator $\RR$.

\end{proof}

\section{The polynomial class $\PP$}
\subsection{Some basic properties of polynomials $P_n$}
Recall that $\PP := \{ P_n: =\RR^n \textbf{1}  \, : \,  n \ge 0 \}$ where the stein operator $\RR$ associated to the favourite random variable $F_\infty$ is given by the second order differential operator $\RR f (x)= x f(x) - f'(x) -x f''(x)$. We start with the following observation on the coefficients of polynomials $P_n$.
\begin{prop}\label{lem:coefficient}
For every $n \ge 1$, the polynomial $P_n \in \PP$ is of degree $n$. Also, assume that 
\begin{equation}\label{eq:poly-n}
 P_{n} (x):=\sum_{k=0}^{n}  \mathfrak{a}(n,k) x^k, \qquad n\ge 1.
 \end{equation} 
Then, for $n \in \N$, the following properties hold.
\begin{enumerate}
\item[(i)] $ \mathfrak{a}(n,n)=1$, i.e. all polynomials $P_n$ are \textit{monic}.
\item[(ii)]  $ \mathfrak{a}(n,n-(2k-1)=0$, for all $ 1 \le k \le [\frac{n+1}{2}]$. In particular, $ \mathfrak{a}(n,n-1)=0$.
\item[(iii)] The doubly indices sequence $\mathfrak{a}(n,k)$ satisfies in the recursive relation 
\begin{equation}\label{eq:recursive}
 \mathfrak{a}(n,k) = \mathfrak{a}(n-1,k-1) - (k+1)^2 \mathfrak{a}(n-1,k+1),
 \end{equation}
  with two terminal conditions $\mathfrak{a}(n,n)=1, \mathfrak{a}(n,n-1)=0$. Moreover, the solution of recursive formula $(\ref{eq:recursive})$, for every $ 0 < k \le \frac{n}{2} (\text{or } \frac{n-1}2)$ depending whether $n$ is even or $n$ is odd, is given by 
  \begin{equation}\label{eq:index-shape}
  \mathfrak{a}(n,n-2k)= (-1)^k \sum_{i_1=1}^{n-2k+1} i^2_1  \sum_{i_2=1}^{i_1+1} i^2_2   \sum_{i_3=1}^{i_2+1} i^2_3 \cdots  \sum_{i_k=1}^{i_{k-1}+1} i^2_k.
  \end{equation}
\item[(iv)] for every $n \in \N$, \begin{equation}
 \mathfrak{a}(n,0) =
\begin{cases}
0, & \text{ if } n \text{ is odd}, \\
(-1)^{n/2} \sum_{i_1=1}^{1} i^2_1  \sum_{i_2=1}^{i_1+1} i^2_2   \sum_{i_3=1}^{i_2+1} i^2_3 \cdots  \sum_{i_{n/2}=1}^{i_{n/2-1}+1} i^2_{n/2}  &,  \text{ if } n \text{ is even}.
\end{cases}
\end{equation}

 \item[(v)] for $n$ even, $$ \mathfrak{a}(n,0)= 4 \sum_{k=1}^{n/2} (-1)^{k-1} \mathfrak{a}(n-2k,2).$$  
 \end{enumerate}
\end{prop} 
 \begin{proof}
 (i) It is straightforward (for example, by an induction argument on $n$) to see that for every $n \in \N$, we have $\text{deg}(P_n) =n$, and moreover $ \mathfrak{a}(n,n)=1$.  (ii) We again proceed with an induction argument on $(n,k)$ for $0 \le k \le n$, and $n \in \N$. Obviously, the claim holds for starting values $n=1,2$. Assume that it holds for some $n \in \N$ that $\mathfrak{a}(n,n-(2k-1)=0$, for all $ 1 \le k \le [\frac{n+1}{2}]$. We want to show it also holds for $n+1$. Using the very definition of polynomial $P_{n+1} (x) = \RR P_n (x) = x P_n (x) -  P'_{n}(x) - x P''_{n}(x)$, and doing some simple computation we infer that  $$ \mathfrak{a}(n+1,k) = \mathfrak{a}(n,k-1) - (k+1)^2 \mathfrak{a}(n,k+1)$$ for $0 \le k \le n+1$ with the convention that $\mathfrak{a}(n,k)=0$ for every value $0 > k > n$. Now the claim for $n+1$ easily follows from induction hypothesis. (iii) Using recursive relation  $(\ref{eq:recursive})$ we can infer that 
 \begin{equation*}
 \begin{split}
 \mathfrak{a}(n,n - 2k) &=  \mathfrak{a}(n-1,(n-1) - 2k) -  (n-2k+1)^2 \mathfrak{a}(n-1,(n-1) - 2(k-1))\\
&=  \Bigg\{   (-1)^k \sum_{i_1=1}^{n-2k} i^2_1  \sum_{i_2=1}^{i_1+1} i^2_2   \sum_{i_3=1}^{i_2+1} i^2_3 \cdots  \sum_{i_k=1}^{i_{k-1}+1} i^2_k  \Bigg\}  \\
 & \qquad -  \Bigg\{  (-1)^{k-1} (n-2k+1)^2 \sum_{i_1=1}^{n-2k+2} i^2_1  \sum_{i_2=1}^{i_1+1} i^2_2   \sum_{i_3=1}^{i_2+1} i^2_3 \cdots  \sum_{i_{k-1}=1}^{i_{k-2}+1} i^2_{k-1}  \Bigg\} \\
 & = (-1)^k \Bigg\{ \Big\{ \sum_{i_1=1}^{n-2k} i^2_1  \sum_{i_2=1}^{i_1+1} i^2_2   \sum_{i_3=1}^{i_2+1} i^2_3 \cdots  \sum_{i_k=1}^{i_{k-1}+1} i^2_k \Big\} \\
 & \qquad + \Big\{ (n-2k+1)^2 \sum_{i_1=1}^{n-2k+2} i^2_1  \sum_{i_2=1}^{i_1+1} i^2_2   \sum_{i_3=1}^{i_2+1} i^2_3 \cdots  \sum_{i_{k-1}=1}^{i_{k-2}+1} i^2_{k-1} \Big\}  \Bigg\} \\
 & = (-1)^k \sum_{i_1=1}^{n-2k+1} i^2_1  \sum_{i_2=1}^{i_1+1} i^2_2   \sum_{i_3=1}^{i_2+1} i^2_3 \cdots  \sum_{i_k=1}^{i_{k-1}+1} i^2_k.
 \end{split}
 \end{equation*}
 (iv) This part is a direct consequence of item (iii), and finally (v) can be obtained directly from the recursive relation $(\ref{eq:recursive})$. 
 \end{proof}

\begin{lem}\label{lem:properties}
  Let $m,n\in \N$. For polynomials $P_n \in \PP$, the following properties hold.
\begin{enumerate}
\item[(i)] $(\RR P'_{n-1}) (x) = P''_{n-1}(x) - P_{n-1}(x) + P'_{n}(x)$.
\item[(ii)] $\RR (P_n P_m)= P_n \RR P_m + P_m \RR P_n -                                                                                                                                                                                                                                                                                                                                                                                                                                                                                                                                                                                                                                                                                                                                                                                                                                                                                                                                                                                                                                                                                                                                                           \left(  2x P'_n P'_m + x P_n P_m \right)$. In particular, $$\RR (xP_n)=x P_{n+1}- P_n -2 x P'_n.$$
 \item[(iii)]$\E[P_n (F_\infty)] = 0$.
 \item[(iv)] for every $n \in \N$, it holds that  $\E[F_\infty P_{n+1} (F_\infty)] = 2 \E[ F_\infty P'_n (F_\infty)]$. In particular, when $n$ is odd both sides vanish. 
 \item[(v)] let $n$ be even, and $m$ odd, or vice versa. Then $\E[P_n(F_\infty)P_m(F_\infty)]=0$. In particular, $$\E[P_n(F_\infty) (\RR P_n)(F_\infty)]= \E[P_n(F_\infty) P_{n+1}(F_\infty)]=0.$$ This item provides some sort of "weak orthogonality".
\end{enumerate}
\end{lem}
\begin{proof}
By very definition of polynomials in the class $\PP$ using raising operator $\RR$ it yields that $P_n = xP_{n-1} - P'_{n-1}-xP''_{n-1}$ for every  $n\ge 1$.  Then items (i), (ii) can be obtained by doing some straightforward computations, using definition of $\RR$ operator, and $P_1 (x)=x$. (iii) It holds that $\E[P_n (F_\infty)] = \E[ \RR(P_{n-1}) (F_\infty)]$, and the later expectation vanishes since the raising operator $\RR$ serves as an associated stein operator for the favourite random variable $F_\infty$. However, to be self-contained, here we present a simple proof base on Lemma \ref{lem:fdgibp}.  Let $u =  (
\begin{matrix}
u_1 \,  u_2\end{matrix} )^t =\frac 1 2 (
\begin{matrix}
N_1 \,  N_2\end{matrix})^t$, so $\delta(u)=(N_1^2 + N_2^2)/2 - 1$. By using the Gaussian integration by parts formula on $L^2( \R^2, \gamma \otimes \gamma)$, where $\gamma$ is the standard Gaussian measure, we write 
\begin{align*} &
 \E\bigl(  f''(N_1 N_2)  N_1 N_2 \bigr) = \E\bigl(  \bigl\langle D f'(N_1N_2), u  \bigr\rangle\bigr)
=  \E\bigl(  f'(N_1 N_2) \delta( u ) \bigr) &\\& = \E\biggl( f'(N_1 N_2) \bigl(  \frac{ N_1^2 + N_2^2 } 2  - 1\bigr) \biggr)
                                                              & \\ &
                                                              =  - \E\bigl( f'(N_1 N_2) \bigr)
      + \E\bigl(  \bigl\langle D f(N_1N_2),v \bigr\rangle \bigr), \quad          v=   \left(
\begin{matrix}
v_1\\ v_2\end{matrix} \right)=\frac 1 2 \left(
\begin{matrix}
N_2\\ N_1\end{matrix} \right)  , & \\ &                                      
                   =      - \E\bigl( f'(N_1 N_2) \bigr) + \E  \bigl( f(N_1 N_2) \delta( v) \bigr)=                                                         
  - \E\bigl( f'(N_1 N_2) \bigr) + \E  \bigl( f(N_1 N_2) N_1N_2 \bigr)   &                                                       
\end{align*}
since $\delta(v)=N_1N_2$ for every function $f$ as soon as the involved expectations exist, in particular the polynomial functions. (iv) It is a direct consequence of Items (ii), (iii), and the fact that $P_n (x) = \RR P_{n-1}(x)$.(v) Note that $\E[P_n(F_\infty)(\RR P_m) (F_\infty)]= \E[ P_n(F_\infty) P_{m+1}(F_\infty)]$, and the later expectation also vanishes relying on item (ii) Lemma \ref{lem:coefficient}, and the fact that the random variable $F_\infty$ is a symmetric distribution yields that all the odd moments  $F_\infty$ vanish. \end{proof}

We ends this section with the non--orthogonality of the polynomials family $\PP$.
\begin{prop} \label{prop:orthogonal} 
The family $\mathscr{P}$ is not a class of orthogonal polynomials.
\end{prop}
\begin{proof}
Using Favard Theorem (see \cite{chihara}), if the family $\PP$ would be a class of orthogonal polynomials, then there exist numerical constants $c_n, d_n$ so that for polynomial $P_n \in \PP$, we have
\begin{equation}\label{eq:1}
P_{n+1}(x) = (x -c_n) P_n (x) - d_n P_{n-1}(x).
\end{equation}
On the other hand, by very definition of operator $\RR$, we have also the following relation
\begin{equation}\label{eq:2}
P_{n+1}(x) = x P_n (x) - P'_n (x) - x P''_n (x).
\end{equation}
Hence, $c_n P_n (x) = P'_n (x) + x P''_n (x) -d_n P_{n-1}(x)$. Now, taking into account that $\text{deg}(P_n)=n$ for all $n \in \N$, when $c_n \neq 0$, a degree argument leads to a contradiction. If $c_{n_0}=0$ for some $n_0 \in \N$, then we have necessary $d_{n_0} =n_0^2$, see also Remark \ref{rem:sheffer+favard}. This is because of the fact that all polynomials $P_n$ are monic.  Now, assume that  
$$P_{n-1}(x) = \sum_{k=0}^{n-1} \mathfrak{a}(n-1,k ) x^k, \qquad \mathfrak{a}(n-1,n-1)=1.$$ Then, using definition of polynomial $P_n$ through of the operator $\RR$, one can obtain that
\begin{equation}\label{eq:pn-1pn}
\begin{split}
P_n(x) &= x^n +\mathfrak{a}(n-1,n-2) x^{n-1} \\
& \quad + \sum_{k=1}^{n-2} \Big( \mathfrak{a}(n-1,k-1) - (k+1)^2 \mathfrak{a}(n-1,k+1) \Big) x^k - \mathfrak{a}(n-1,1).
\end{split}
\end{equation}
On the other hand, relation $n^2_0 P_{n_0 -1}(x) =  P'_{n_0} (x) + x P''_{n_0} (x)$ implies that $P_{n_0+1}(x)= x P_{n_0}(x) - n^2_0 P_{n_0-1}$. Hence, 
$$\mathfrak{a}(n_0 +1,0) = 4  \mathfrak{a}(n_0 -1,2)  - \mathfrak{a}(n_0-1,0)=n^2_0  \mathfrak{a}(n_0-1,0),$$ i.e. $4 \mathfrak{a}(n_0-1,2) = (1 +n^2_0) \mathfrak{a}(n_0-1,0)$, which also leads to a contradiction, because one side is positive and the other side negative.
\end{proof}

\begin{rem} { \rm
It is worth noting that, it is easy to see the class $\PP$ is not orthogonal with respect to probability measure induced by random variable $F_\infty$. For example, we have $\E[P_2(F_\infty) \times P_4 (F_\infty)] = 94 \neq 0$. See also item (v) of the forthcoming lemma. 
}
\end{rem}

We close this section with a neat application of the adjoint operator $\RR^\star$. We present the following average version of the well known \textit{Turan's inequality} in the framework of orthogonal polynomials.
\begin{prop}\label{lem:turan}
For every $n \in \N$, the following inequality hold
\begin{equation}\label{eq:turan}
\E[P^2_n(F_\infty)] \ge \E[P_{n-1}(F_\infty) P_{n+1}(F_\infty)].
\end{equation}
\end{prop}

\begin{proof}
According to Proposition \ref{prop:adjoint}, we can write
\begin{equation*}
\begin{split}
\E & [P^2_n (F_\infty)]  = \E[P_n(F_\infty) \LL P_{n-1}(F_\infty)]
= \E[ P_{n-1}(F_\infty) \LL^\star P_n (F_\infty)]\\
&= \E[P_{n-1}(F_\infty) \LL P_n (F_\infty)] +\E[P_{n-1}(F_\infty) \theta(F_\infty) P'_n (F_\infty)]  
-\E[F_\infty P_{n-1}(F_\infty) P_n(F_\infty)]\\
&=  \E[P_{n-1}(F_\infty) P_{n+1}(F_\infty)]  +\E[P_{n-1}(F_\infty) \theta(F_\infty) P'_n (F_\infty)]  
 -\E[F_\infty P_{n-1}(F_\infty) P_n(F_\infty)],
\end{split}
\end{equation*}
where the special function $\theta (x) = 2 \vert x \vert \frac{K_1(\vert x \vert )}{K_0 (\vert x \vert)}$. Hence, we are left to show that $$\E[P_{n-1}(F_\infty) \theta(F_\infty) P'_n (F_\infty)]  \ge \E[F_\infty P_{n-1}(F_\infty) P_n(F_\infty)].$$ The later is equivalent to show that 

\begin{equation}\label{eq:inequality}
2 \int_0^\infty x P_{n-1}(x) P'_n(x) K_1 (x) dx \ge \int_0^\infty  x P_{n-1}(x) P_n(x) K_0 (x) dx.
\end{equation}

 We need also the following integral formula is taken from Gradshetyn and Ryzhik \cite[page 676]{integral-formula}
 \begin{equation}\label{eq:integral-formula}
\int_0^\infty x^\mu K_\nu (x) dx = 2^{\mu -1} \Gamma(\frac{1+\mu +\nu}{2}) \Gamma(\frac{1+\mu -\nu}{2}), \quad \Re(\mu+1 \pm \nu) >0.
\end{equation}
Now assume that $P_n (x) = \sum_{k=0}^{n} \mathfrak{a}(n,k) x^k$. Then using integral formula $(\ref{eq:integral-formula})$ together with some straightforward computation, inequality $(\ref{eq:inequality})$ is equivalent with showing that 
\begin{multline*}
2 \sum_{k=1}^{2n-1} \sum_{l=0}^{k} l \mathfrak{a}(n,l) \mathfrak{a}(n-1,k-l)   2^{k-1} \Gamma(\frac{2+k}{2}) \Gamma(\frac{k}{2}) \\
=  2 \sum_{k=1}^{2n-1} \sum_{l=0}^{k} l \mathfrak{a}(n,l) \mathfrak{a}(n-1,k-l) 2^{k-1} k \Big[   \frac{(k-2)!! \sqrt{\pi}}{2^{\frac{k}{2}}}  \Big]^2 \\
\ge \sum_{k=1}^{2n} \sum_{l=0}^{k} \mathfrak{a}(n,l) \mathfrak{a}(n-1,k-l-1)  2^{k-1}\Gamma(\frac{1+k}{2}) \Gamma(\frac{1+k}{2})\\
=  \sum_{k=1}^{2n} \sum_{l=0}^{k} \mathfrak{a}(n,l) \mathfrak{a}(n-1,k-l-1)  2^{k-1} \Big[  \frac{(k-1)!! \sqrt{\pi}}{2^{\frac{k}{2}}}   \Big]^2
 \end{multline*}
 We set 
 \begin{align*}
 A(n,k) = (k!!)^2 \sum_{l=0}^{k} \frac{l}{k} \mathfrak{a}(n,l) \mathfrak{a}(n-1,k-l), \quad 
 B(n,k) = \frac{(k!!)^2}{2}  \sum_{l=0}^{k} \mathfrak{a}(n,l) \mathfrak{a}(n-1,k-l).
 \end{align*}
 Hence, we are left to show that $$\sum_{1 \le k \le 2n-1, \\  k \text{ odd }} A (n,k ) \ge \sum_{1 \le k \le 2n-1, \\ k \text{ odd }} B (n,k ).$$ The last inequality itself, can be shown,  using induction on $n$, together with some straightforward computations but tedious, the recursive relation $(\ref{eq:recursive})$, and the shape of coefficients $\mathfrak{a}(n,k)$ given by relation $(\ref{eq:index-shape})$.
\if
Using integration by parts formulae, $K'_0 (x) = - K_1 (x)$, the asymptotic behaviour  $K_0 (x) \sim \sqrt{\frac{\pi}{2x}} e^{-x}$ as $x \to \infty$, we infer that 

\begin{equation*}
\begin{split}
\E[P_{n-1} & (F_\infty)  \theta(F_\infty) P'_n (F_\infty)]  \\
&= 2 \E [ P_{n-1}(F_\infty) P'_{n}(F_\infty) + F_\infty P'_{n-1}(F_\infty) P'_n (\infty) + F_\infty P_{n-1}(F_\infty) P''_n (F_\infty) ]\\
&= -2 \E[P_{n-1}(F_\infty) \RR P_n (F_\infty)] + 2 \E[F_\infty P_{n-1} (F_\infty)  P_n (F_\infty) ] +2 \E[F_\infty P'_{n-1}(F_\infty) P'_n (\infty) ]\\
&=  -2 \E[P_{n-1}(F_\infty) P_{n+1} (F_\infty)] + 2 \E[F_\infty P_{n-1} (F_\infty)  P_n (F_\infty) ] +2 \E[F_\infty P'_{n-1}(F_\infty) P'_n (\infty) ]
\end{split}
\end{equation*} 
So, we are left to show that $$\E[F_\infty P_{n-1} (F_\infty)  P_n (F_\infty) ] +2 \E[F_\infty P'_{n-1}(F_\infty) P'_n (\infty) ] \ge 0.$$ Now, using item (ii) with $m=n-1$ of Lemma \ref{lem:properties} we obtain that 
$$\E[F_\infty P_{n-1} (F_\infty)  P_n (F_\infty) ] +2 \E[F_\infty P'_{n-1}(F_\infty) P'_n (\infty) ] = \E[P^2_n (F_\infty)] + \E[P_{n-1}(F_\infty) P_{n+1}(F_\infty)].$$ \fi

\end{proof}

\subsection{Generating function and the Sheffer sequences}\label{sec:sheffer}

 In this section, we provide some fundamental elements of theory of Sheffer class polynomials. For a complete overview, the reader may consult the monograph \cite{Roman-book}. Sequences of polynomials play a fundamental role in mathematics. One of the most famous classes of polynomial sequences is the class of Sheffer sequences, which contains many important sequences such as those formed by Bernoulli polynomials, Euler polynomials, Abel polynomials, Hermite polynomials, {\it Laguerre} polynomials, {\em etc.} and contains the classes of associated sequences and Appell sequences as two subclasses. Roman {\em et. al.} in \cite{Roman-book, RomanRota} studied the Sheffer sequences systematically by the theory of modern umbral calculus.\\


Let $\mathbb{K}$ be a field of characteristic zero. Let $\mathcal{F}$ be the set of all formal power series in the variable $t$ over $\mathbb{K}$.
Thus an element of $\mathcal{F}$ has the form
$$f(t)=\sum\limits_{k=0}^\infty a_k t^k,\eqno(1.1)$$
where $a_k \in \mathbb{K}$ for all $k \in \mathbb{N}:=\{0,1,2,\ldots\}$. The order O$(f(t))$ of a power series $f(t)$ is the smallest integer $k$
for which the coefficient of $t^k$ does not vanish. The series $f(t)$ has a multiplicative inverse, denoted by ${f(t)}^{-1}$ or $\frac{1}{f(t)}$,
if and only if O$(f(t))=0$. Then $f(t)$ is called an invertible series. The series $f(t)$ has a compositional inverse, denoted by $\bar{f}(t)$
and satisfying $f(\bar{f}(t))=\bar{f}(f(t))=t$, if and only if  O$(f(t))=1$. Then $f(t)$ is called a delta series. Let $f(t)$ be a delta series and $g(t)$ be an invertible series of the following forms:
$$f(t)=\sum\limits_{n=0}^\infty f_n \frac{t^n}{n!},~~~f_0=0,~f_1\neq0\eqno(1.2a)$$
and
$${g}(t)=\sum\limits_{n=0}^\infty g_n \frac{t^n}{n!},~~~g_0\neq0.\eqno(1.2b)$$

\begin{thm}(Sheffer sequence \cite[Theorem 2.3.1]{Roman-book})
Let $f(t)$ be a delta series and let $g(t)$ be an invertible series. Then there exists a unique sequence $s_n(x)$ of polynomials satisfying the orthogonality conditions
\begin{equation}\label{eq:sheffer-orthogonal}
\langle g(t){f(t)}^k |s_n(x) \rangle=c_n \delta_{n,k},
\end{equation}
for all $n,k \geq 0$, where $\delta_{n,k}$ is the kronecker delta, and $\langle L \, \vert \, p(x) \rangle $ denote the action of a linear functional $L$ on a polynomial $p(x)$. In this case, we say that the sequence $s_n(x)$ in $(\ref{eq:sheffer-orthogonal})$ is the Sheffer sequence for the pair
$(g(t), f(t))$, or that $s_n(x)$ is Sheffer for $(g(t), f(t))$. In particular, the Sheffer sequence for $(1,f(t))$ is called the associated sequence
for $f(t)$ defined by the generating function of the form
$$e^{x\bar{f}(t)}=\sum\limits_{n=0}^\infty \tilde{s}_n(x)\frac{t^n}{n!}\eqno(1.6)$$
and the Sheffer sequence for $(g(t),t)$ is called the Appell sequence for $g(t)$ defined by the generating function of the form
$$\frac{1}{g(t)}e^{xt}=\sum\limits_{n=0}^\infty A_n(x)\frac{t^n}{n!}.\eqno(1.7)$$
\end{thm}

\begin{thm}(\cite[Theorem 2.3.4]{Roman-book})\label{thm:sheffer-gf}
 The sequence $s_n(x)$ in equation $(\ref{eq:sheffer-orthogonal})$ is the Sheffer sequence for the pair $(g(t),f(t))$ if and only if they admit the exponential generating function of the form
\begin{equation}\label{eq:sheffer-cf}
\frac{1}{g(\bar{f}(t))}e^{x(\bar{f}(t))}=\sum\limits_{n=0}^\infty s_n(x)\frac{t^n}{n!},
\end{equation}
 where $\bar{f}(t)$ is the compositional inverse of $f(t)$.
\end{thm}

Now, consider the class $\PP$ of polynomials $P_n$. Let $G$ denotes the generating function, i.e.
$$G(t,x):= \sum_{n \ge 0} P_n(x)\frac{t^n}{n!}.$$

\begin{thm}\label{thm:sheffer}
For the random variable $F_\infty = N_1 \times N_2$ consider the associated family of polynomials $\PP$ given by $(\ref{eq:main-poly})$. Then, the family $\PP$ is Sheffer for the pair  $$(f(t),g(t))=\Big( \coth^{-1} (\frac{1}{t}), \frac{1}{\sqrt{1-t^2}}\Big).$$ 
\end{thm}
\begin{proof}
 Note that the generating function $G$ satisfies in the following second order PDE $$\frac{d}{dx} \left( x \frac{d}{dx} G(t,x) \right) = x G(t,x)- \frac{d}{dt} G(t,x),$$ i.e. 
 \begin{equation}\label{eq:cf-pde}
G_t = x G - G_x - x G_{xx}.
\end{equation}
 Hence, the PDE $(\ref{eq:cf-pde})$ is just the associated PDE, scaled in space, in the Feynman-Kac formula for the squared Bessel process with index $\delta=2$, and therefore \cite[Theorem 5.4.2]{Platen-book} immediately implies that 
 \begin{equation*}
G(t,x)= \frac{e^{\frac{x}{\coth (t)}}}{\cosh (t)}.
\end{equation*}
Therefore $f(t)= \coth^{-1}(\frac{1}{t})$ in the representation $(\ref{eq:sheffer-cf})$, and moreover, $g(t)= \frac{1}{\sqrt{1-t^2}}$ is a direct consequence of the hyper trigonometric identity $\sinh ( \cosh^{-1} (t)) = \sqrt{t^2 -1}$.
\end{proof}

A polynomial set $\PP= \{P_n\}_{n \ge 0}$ (i.e. $\text{deg} (P_n) =n$ for all $n \ge 0$) is said to be {\it quasi--monomial} if there are two operators $\RR$, and $\LL$ independent of $n$ such that 
\begin{equation}\label{eq:quasi-monomial}
\RR (P_n) (x) = P_{n+1}(x), \quad \text{and} \quad \LL(P_n)(x)= P_{n-1}(x).
\end{equation}
In other words, operators $\RR$ and $\LL$ play the similar roles to the multiplicative and derivative operators, respectively, on monomials.  We refer to the $\LL$ and $\RR$ operators as the {\it descending} (or {\it lowering}) and {\it ascending} (or {\it raising}) operators associated with the polynomial set $\PP$. A fundamental result (see \cite[Theorem 2.1]{cheikh1}) tells that every polynomial set is quasi--monomial in the above sense. The next corollary aims to take the advantage of being Sheffer the polynomial set $\PP$  to provide an explicit form of the associated lowering operator $\LL$.

\begin{cor}\label{col:lowering-operator}
For the Sheffer family $\PP$ of polynomials $P_n$ associated to random variable  $F_\infty = N_1 \times N_2$, the lowering operator, i.e. $\LL P_n = n P_{n-1}$ for $n \ge 1$ is given by 

\begin{equation*}
 \LL = f (D), \quad \text{and} \quad  f (t) = \coth^{-1} (\frac{1}{t})= \sum_{k \ge 0} \frac{t^{2k+1}}{2k+1},
\end{equation*}
where $D$ stands for derivative operator.
\end{cor}

\begin{proof}
This is an application of \cite[Theorem 2.3.7]{Roman-book}.
\end{proof}
The polynomial set $\PP$ appearing in Corollary \ref{col:lowering-operator} is called $\LL$- {\it Appell polynomial set} since $$\LL P_n = n P_{n-1}, \quad \forall \, n \ge 1.$$ This notion generalizes the classical concept of the Appell polynomial set meaning that $\frac{d}{dx} P_n (x) = n P_{n-1}(x)$ for every $ n\ge 1$. Among the classical Appell polynomials set, we recall the monomials set $\{x^n\}_{n \ge 0}$, {\it Hermite polynomials}, the {\it Bernoulli polynomials}, and the {\it Euler polynomials}. For the application of Appell polynomials in noncentral probabilistic limit theorems, we refer the reader to \cite{murad}.

The next corollary provides  more information on the constant coefficients of the even degree polynomials $P_n \in \PP$.
\begin{cor}\label{cor:0coefficients}
 Let $\PP = \{ P_n \, : \, n \ge 1 \}$ be the associated Sheffer family of the random variable $F_\infty=N_1 \times N_2$. Assume that $E_n, n \ge 1$ stands for Euler numbers (see \cite[Section 1.14, page $48$]{Comtet-book}). Then
 $$ \mathfrak{a}(2n,0 ) = E_{2n}, \quad \ n \ge 1.$$
 Hence, the following representation of even Euler numbers is in order, for $n \in \N$ even number,
 \begin{equation}\label{eq:Euler-new?}
 E_n = (-1)^{n/2} \sum_{i_1=1}^{1} i^2_1  \sum_{i_2=1}^{i_1+1} i^2_2   \sum_{i_3=1}^{i_2+1} i^2_3 \cdots  \sum_{i_{n/2}=1}^{i_{n/2-1}+1} i^2_{n/2}.
 \end{equation}
\end{cor}

\begin{proof}
According to \cite[Theorem 2.3.5]{Roman-book}, $$P_n(x)= \sum_{k=0}^{n} \frac{1}{k!} \big \langle g (\bar{f}(t))^{-1} f(t)^k \big \vert x^n \big \rangle x^k.$$ Hence, the result follows at once by taking into account Theorem \ref{thm:sheffer}, and  Taylor series expansion of Euler numbers $$ \frac{1}{\cosh (t)} = \frac{2e^t}{e^{2t}+1}= \sum_{n \ge 0} E_n \frac{t^n}{n! }.$$
\end{proof}

\subsection{Orthogonal Sheffer polynomial sequences}
There are several charactrizations of the orthogonal Sheffer polynomial sequences. Here we mentioned the one in terms of their generating function originally due to Sheffer (1939) \cite{sheffer-orthogonal}. Consider a Sheffer polynomial sequence $\{s_n(x)\}_{n\ge 0}$ associated to the pair $(g(t),f(t))$, see Theorem \ref{thm:sheffer-gf}, with the generating function 
\begin{equation}\label{eq:gf}
G(t,x) = \frac{1}{g(\bar{f}(t))}e^{x(\bar{f}(t))}=\sum\limits_{n=0}^\infty s_n(x)\frac{t^n}{n!}.
\end{equation}
\begin{thm}\label{thm:sheffer-othogonal}
A Sheffer polynomial sequence $\{s_n(x)\}_{n\ge 0}$ is orthogonal if and only if its generating function $G(t,x)$ is one of the following forms:
\begin{align}
G(t,x)&= \mu (1-bt)^c \exp\{ \frac{d + at  x}{1-bt} \}, \quad abc\mu \neq 0,\\
G(t,x)&= \mu \exp\{ t (b + ax) + ct^2\}, \quad ac\mu \neq 0,\\
G(t,x)&= \mu e^{ct} (1-bt)^{d +ax}, \quad abc\mu \neq 0,\\
g(t,x)&= \mu (1- \frac{t}{c})^{d_1 + \frac{x}{a}} (1- \frac{t}{b})^{d_2 - \frac{x}{a}}, \quad abc\mu \neq 0, b \neq c. 
\end{align}
\end{thm}
Among the well known Sheffer orthogonal polynomial sequences are {\it Laguerre, Hermite, Charlier, Meixner, Meixner--Pollaczek} and {\it Krawtchouk polynomials}. See the excellent textbooks \cite{ismail,chihara} for definitions and more information. Theorem \ref{thm:sheffer-othogonal}  can be directly performed to give an alternative proof of Proposition \ref{prop:orthogonal}. The reader is also refereed to references \cite{macdonald,cheikh2} for related results on the associated orthogonal polynomials with respect to the probability measure $F_\infty = N_1 \times N_2$ on the real line. 

\begin{cor}
Let $F_\infty = N_1 \times N_2$, where $N_1, N_2 \sim \mathscr{N}(0,1)$ are independent, and the associated polynomials set $\PP$ is given by $(\ref{eq:main-poly})$. Then polynomial family $\PP$ is Sheffer but not orthogonal. 
\end{cor}

\begin{rem}\cite{sheffer-orthogonal}\label{rem:sheffer+favard}{ \rm A necessary and sufficient condition for $\{s_n(x)\}_{n \ge 0}$ to be a orthogonal Sheffer family is that the monic recursion coefficients $a_n$ and $b_{n}$ in the three-term recurrence relation $s_{n+1}(x) = (x - a_n) s_n (x) - b_n s_{n-1}(x)$ have the form
$$a_{n+1} =c_1 +c_2 \, n, \quad \text{ and } \quad b_{n+1} =c_3 \, n+c_4 \, n^2, \qquad  c_1,\cdots,c_4 \in \R$$  with $b_{n+1} > 0$, in oder words $a_n$ is at most linear in $n$ and that $b_n$ is at most
quadratic in $n$.

}
\end{rem}

\section{Connection with weak convergence on the second Wiener chaos}

\subsection{Normal approximation with higher even moments}
The aim of this section is to build possibly a bridge between the Sheffer family of polynomials $\PP$ given by $(\ref{eq:main-poly})$ associated to target random variable $F_\infty = N_1 \times N_2$ and the non--central limit theorems on the second Wiener chaos.  We start with a striking result appearing in $2005$ known nowadays as the {\it fourth moment Theorem} due to Nualart \& Peccati  \cite{n-p-4m} stating that for a normalized sequence $\{F_n\}_{n \ge 1}$, meaning that $\E[F^2_n]=1$ for all $ n\ge 1$, in a fixed Wiener chaos of order $p \ge 2$, the weak convergence $F_n$ towards $\mathscr{N}(0,1)$ is equivalent with convergence of the fourth moments $\E[F^4_n] \to 3 (= \E[\mathscr{N}(0,1)^4])$.   In the case of normal approximation on the Wiener chaoses, the authors of \cite{a-m-m-p} introduced a novel family of special polynomials that characterizes the weak convergence of the sequence $\{F_n\}_{n \ge 1}$ towards standard normal distribution.  Assume that $\{F_n\}_{n \ge 1}$ be a sequence  of random elements in the fixed Wiener chaos of order $p \ge 2$ such that $\E[F^2_n]=1$ for all $ n\ge 1$. Following \cite{a-m-m-p} consider the family $\mathscr{Q}_{\ge 0}$ of polynomials defined as follows: for any $k \ge2$, define the monic polynomial $W_k$ as
\begin{equation}\label{Wk}
W_k(x) = (2k-1) \Big( \ x \int_{0}^{x} H_{k}(t)H_{k-2}(t) \ud t - H_{k}(x)H_{k-2}(x)\Big),
\end{equation}
where $H_k$ is the $k$th Hermite polynomial, and 
\begin{multline}\label{eq:polynomial-family}
\mathscr{Q}_{\ge 0} := \Big\{  P \ \Big\vert \ P(x)= \sum_{k=2}^{m} \alpha_k W_k(x); \ m \ge 2, \  \alpha_k \ge 0 \ \,,\ 2\le k \le m \Big\}.
\end{multline}
Then, one of the main findings of \cite{a-m-m-p} is that the polynomial family $\mathscr{Q}_{\ge 0}$ characterizes the normal approximation on the Wiener chaoses in the sense that the following statements are equivalent:
\begin{description}\label{eq:Qnormal}
\item[(I)]$F_n \stackrel{\text{law}}{\longrightarrow} \mathscr{N}(0,1)$.
\item[(II)] $(0 \le ) \, \E[P(F_n)]\to \E[P(\mathscr{N}(0,1))] \, (=0) \text{ {\bf for some }} P \in  \mathscr{Q}_{\ge 0}$.
\end{description}
One of the significant consequences of the aforementioned equivalence is the following generalization of the Nualart--Peccati fourth moment criterion. 

\begin{thm}[{\bf Even moment Theorem} \cite{a-m-m-p}]
Let $\{F_n\}_{n \ge 1}$ be a sequence  of random elements in the fixed Wiener chaos of order $p \ge 2$, and that $\E[F^2_n]=1$ for all $ n\ge 1$.  Let $N \sim \mathscr{N}(0,1)$. Then the following asymptotic statements are equivalent:
\begin{description}\label{eq:Qnormal}
\item[(I)]$F_n \stackrel{\text{law}}{\longrightarrow}N$.
\item[(II)] $m_{2k}(F_n):= \E [ F^{2k}_n] \longrightarrow m_{2k}(N)= (2k-1)!!$.
\end{description}
Furthermore, for some constant $C$, independent of $n$, the following estimate in the total variation probability metric takes place 
$$d_{TV}(F_n,\mathscr{N}(0,1) ) \le_C \sqrt{ \E [ F^{2k}_n] -  (2k-1)!!}.$$
\end{thm}

\subsection{Convergence towards $N_1 \times N_2$: cumulants criterion}
We start with the following observation that connect the polynomials class $ \PP$ with some non--central limit theorems on the Wiener chaoses when the target distribution $F_\infty = N_1 \times N_2$, equality in distribution. In the sequel $D$ and $L \, (L^{-1})$ stand for the Malliavin derivative operator, the infinitesimal Ornstein-Uhlenbeck generator (the pseudo-inverse of operator $L$) respectively.  Next, we define the {\it iterated Gamma operators} (see the excellent monograph \cite{n-p-book} for a complete overview on the topic as well as the non--explained notations) as follows: for a `'smooth'' random variable $F$ in the sense of Malliavin calculus, define $\Gamma_0(F)=F$, and $\Gamma_r(F)= \langle DF, - D L^{-1}\Gamma_{r-1}(F) \rangle_{\HH}$  for $r \ge 1$, where $\HH$ is the underlying separable Hilbert space. Also, the useful fact $\kappa_r(F) = (r-1)! \E[\Gamma_{r-1}(F)]$ is well known, see \cite[Theorem 8.4.3]{n-p-book} where $\kappa_r (F)$ stands for the $r$th cumulant of the random variable $F$. \\

We continue with the following non--central convergence towards the target distribution $F_\infty = N_1 \times N_2$ in terms of the convergences of finitely many cumulants/moments. In the framework of the second Wiener chaos, it has been first proven in \cite{n-p-2w} using the method of complex analysis. For  a rather general setup using the iterated Gamma operators and the Malliavin integration by parts formulae, see \cite{a-p-p}. Also, for quantitative  Berry--Essen estimates see the recent works \cite{e-t, a-a-p-s-2w}, and \cite{a-g,d-n} for the free counterpart statements.  

\begin{thm}\label{thm:2w-requirments}
Let $\{ F_n \}_{n \ge 1}$ be a centered sequence of random elements in a finite direct sum of Wiener chaoses such that $\E[F^2_n]=1$ for all $n \ge 1$. Then, the following statements are equivalent. 
\begin{description}
\item[(I)] sequence $F_n$ converges in distribution towards $F_\infty \sim  N_1 \times N_2$.
\item[(II)] as $n \to \infty$ the following asymptotic relations hold:
\begin{enumerate}
 \item $\kappa_3(F_n) \to 0$.
 \item $\Delta(F_n):=   \text{Var} \left( \Gamma_2 (F_n) - F_n \right) \to 0$.
\end{enumerate}
\end{description}
Whenever the sequence $\{ F_n \}_{n \ge 1}$ belongs to the second Wiener chaos, the quantity $\Delta(F_n)$ appearing in item $2$ can be replaced with 
\begin{equation}\label{eq:Delta}
\Delta'(F_n)=  \frac{\kappa_6(F_n)}{5!} - 2\frac{ \kappa_4(F_n)}{3!}+ \kappa_2(F_n).
\end{equation}
\end{thm}
 
The following proposition aims to provides a direct link between the Sheffer polynomial class $\PP$ and Theorem \ref{thm:2w-requirments}. The {\it Wasserstein$-2$ distance} between two probability distributions $Q_1,Q_2$ on $(\R,{\mathcal B}(\R) )$ is given by
\begin{align*}
  d_{W_2}( Q_1,Q_2):=\inf_{(X_1,X_2) } \biggl\{  \E\biggl(  (X_1-X_2)^2\biggr)^{1/2}   \biggr\}
\end{align*}
where the supremum is taken over the  random pairs  $(X_1,X_2)$ defined on the same classical probability spaces $(\Omega,{\mathcal F},\P)$
with marginal distributions $Q_1$ and $Q_2$.
\begin{prop}\label{prop:p6}
Let $\{ F_n \}_{n \ge 1}$ be a sequence in the second Wiener chaos such that $\E[F^2_n]=1$ for all $n \ge 1$. Consider polynomial $ P_6(x)= x^6 - 55 \, x^4 +331 \, x^2 - 61 \in \PP$. Then, as $n$ tends to infinity, the following statements are equivalent. 
\begin{description}
\item[(I)] sequence $F_n$ converges in distribution towards $F_\infty \sim N_1 \times N_2$.
\item[(II)] $\E[F^4_n] \to 9$, and $\E[F^6_n] \to 225$.
\item[(III)]$(0 \le ) \, \E[P_6 (F_n)] \to \E[P_6 (F_\infty)] \, (=0)$.

\end{description}
In other words, polynomial $P_6$ captures at the same time the two necessary and sufficient conditions for convergence towards $F_\infty$ appearing in Theorem \ref{thm:2w-requirments}. Furthermore, the following quantitive estimate in Wasserstein$-2$ distance holds: for $n \ge 1$, 
\begin{equation}\label{eq:2w}
d_{W_2}(F_n,F_\infty) \le_C \sqrt{P_6(F_n)} \le_C \sqrt{ \Big(  \E[F^6_n] - 225   \Big)   - 55 \Big(   \E[F^4_n] - 9 \Big)  }.
\end{equation}
\end{prop}

\begin{proof}
Implication ${\bf (I)} \Rightarrow {\bf (II)}$ is just an application of the continuous mapping theorem. ${\bf (II)} \Leftrightarrow {\bf (III)}$  Using the relation between moments and cumulants of random variables (see \cite[page $259$]{p-t-book}) and straightforward computation, and taking into account that $\E[P_6 (F_\infty)]=0$, it yields that
\begin{equation}\label{eq:p6>0}
\begin{split}
 \E[P_6 (F_n)] &=   \big(  \E[F^6_n] - 225   \big)   - 55 \big(   \E[F^4_n] - 9 \big)\\
 &= 5! \Delta'(F_n) + 10 \kappa^2_3 (F_n) \ge 0\\
 \end{split}
\end{equation}
under the lights of Item $2$ at Theorem \ref{thm:2w-requirments}, and the fact that $F_n$ being in the second Wiener chaos. Finally implication ${\bf (II)} \Rightarrow {\bf (I)}$ together with the estimate $(\ref{eq:2w})$ is borrowed from \cite[Proposition 5.1]{a-g}. 
\end{proof}

\begin{rem}\label{rem:P6P8}{ \rm The following remarks are of independent interests. Let  $F$ be a random element in the second Wiener chaos with $\E[F^2]=1$. (a)  The crucial relations
\begin{equation*}
\begin{split}
(0 \le ) \, \E[P_6(F)]& = 5!  \Delta'(F) + 10 \kappa^2_3 (F)\\
& = 5!   \text{Var} \left( \Gamma_2 (F) - F \right)+ 10 \kappa^2_3 (F).
\end{split}
\end{equation*}
in Proposition \ref{prop:p6} can be deduce from using the Malliavin integration by part formulae instead the relation between moments and cumulants. In fact, using very definition of polynomials in the family $\PP$ thorough the rising operator $\RR$, and also applying twice Malliavin integration by parts formula \cite[Theorem 2.9.1]{n-p-book}, we can write (to follow incoming computation, one has to note that $\E[F]=0$, and $\E[\Gamma_1(F)]=\E[F^2]=1$)

\begin{equation}\label{eq:p6>0}
\begin{split}
\E[P_6(F)] & = \E[ \RR P_5 (F )] = \E[F P_5(F) - P'_5(F) - F P''_5(F)]\\
&=\E \left[P''_5(F) \left( \Gamma_2(F) - \E[\Gamma_2(F)] -F \right) \right] + \E[\Gamma_2(F)] \, \E[P''_5(F)]\\
&= \E \left[P''_5(F) \left( \Gamma_2(F) - \E[\Gamma_2(F)] -F \right) \right] + 10 \kappa^2_3(F).
\end{split}
\end{equation}
As a direct consequence, in order to have $ \text{Var} \left( \Gamma_2 (F) - F \right)$ in the very last right hand side of relation $(\ref{eq:p6>0})$, there must be one more copy of the random variable $\Gamma_2(F) - \E[\Gamma_2(F)] -F$ inside the quantity $P''_5(F)$. Now, note that for $F=I_2(f)$ being in the second Wiener chaos, $$\Gamma_r(F) - \E[\Gamma_r(F)] = I_2 (2^r f \otimes^{(r+1)}_1 f) \quad r \ge 1,$$ in which implies that random variable $\Gamma_2(F) - \E[\Gamma_2(F)] -F $ belongs to the second Wiener chaos. Now, taking into account orthogonality of the Wiener chaoses, in order to compute $\E \left[P''_5(F) \left( \Gamma_2(F) - \E[\Gamma_2(F)] -F \right) \right]$, one needs only to understand the projection of random variable $P''_5(F)$ on the second Wiener chaos. Since, $F$ is a multiple integral, and $P''_5$ is a polynomial, so random variable $P''_5(F)$ is smooth in the sense of Malliavin differentiability. Hence, one can use Stroock's formula \cite[Corollary 2.7.8]{n-p-book} to compute the second projection. We have $$P_5 (x)= x^5 - 30 x^3 + 61 x, \quad \Rightarrow \quad P''_5 (x)= 20 (x^3 - 9 x).$$ Hence, $P''_5(F)=\sum_{p=0}^{3} I_p (g_p)$, where $$g_2 (t_1,t_2)=\frac{1}{2!} \E \left[  6 F( D_{t_1}F) ( D_{t_2}F) + (3F^2 - 9) D^2_{t_1,t_2} F  \right].$$ For example,

\begin{equation*}
\begin{split}
 \E \left[  6 F( D_{t_1}F) ( D_{t_2}F)\right] & =6 \times 2 \times 2 \E \left[ I_2 (f) \times I_1(f(t_1,.)) \times I_1 (f(t_2,.)) \right]\\
 &= 24 \E \left[ I_2 (f) \times I_2 \left(  f(t_1,.) \otimes f(t_2,.) \right) \right]\\  
 & = 24 \times 2 \langle f,  f(t_1,.) \otimes f(t_2,.) \rangle \\
 &= 6 \times 2^3 (f \otimes^{(3)}_1 f) (t_1,t_2).
\end{split}
\end{equation*}
The similar computations can be done for the other term. All together imply that $$g_2 (t_1,t_2)= 5! \left( 2^2 (f \otimes^{(3)}_1 f) (t_1,t_2) - f(t_1,t_2) \right).$$ The later is exactly the kernel of random element $\Gamma_2(F) - \E[\Gamma_2(F)] -F$.\\
(b) Assume that $F_\infty \sim N_1 \times N_2$. Then
\begin{equation*}
\begin{split}
x^8 - 11025 & =x^8 - \E [F^8_\infty] \\
&= P_8(x) + 140 \, P_6(x) + 4214 \,P_4 (x) + 24940 \, P_2 (x).
\end{split}
\end{equation*}
This in turns implies that $$\Big\{  m_8 (F ) - m_8 (F_\infty) \Big\}- 4214 \Big\{  m_4 (F ) - m_4 (F_\infty) \Big\} = \E[P_8 (F)] + 140 \E[P_6 (F)].$$ However, in general, for a random element $F$ in the second Wiener chaos with $\E[F^2]=1$, unlike the quantity $\E[P_6 (F)]$, the expectation $\E[P_8 (F)]$ can take negative values too. For example, assume that $G_\infty$ is an independent copy of $F_\infty$. Consider, for every $t \in [0,1]$, the random element $$F_t:= \sqrt{t} F_\infty + \sqrt{1-t} G_\infty.$$ Note that $F_t$ belongs to the second Wiener chaos, and that $\text{law}(F_t) \neq \text{law}(F_\infty)$ for every $t \in (0,1)$.  Define the auxiliary function $Q_4(t):= \E[ P_8 ( F_t)]$. Using MATLAB, we obtain that $$Q_4(t) =  15120\, t^4 - 30240\, t^3 + 19152\, t^2 - 4032\, t.$$  The graph of the polynomial $Q_4$ is shown in Figure \ref{fig:P8}. We note that $Q_4 (t) \le 0$ for every $t \in [0,1]$, and furthermore $Q_4 (0) = \E[P_8(G_\infty)] = Q_4 (1) = \E[ P_8 (F_\infty)]=0$. As a conclusion, this line of argument cannot be useful to justify the fact that convergences of the fourth and the eighth moments are enough to declare convergence in distribution. See \cite{a-g}, and also the forthcoming proposition. 

\begin{figure}\label{fig:P8}
 \centering 
 \includegraphics[scale=0.5]{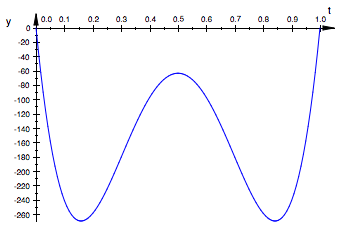}
 \caption{ $Q_4(t)=15120\, t^4 - 30240\, t^3 + 19152\, t^2 - 4032\, t$}
 \label{fig:P8}
\end{figure}
}
\end{rem}

\begin{definition}\label{def:charactrize}
Let  $F_\infty \sim N_1 \times N_2$. We say that a polynomial $P$ (of degree $\ge 3$) characterizes the law of $F_\infty$ whenever $\E[P(F)]=0$ for some random element $F$ inside the second Wiener chaos so that $\E[F^2]=1$ then $F \sim F_\infty$. Also, we say that polynomial $P$ sequentially characterizes the law of $F_\infty$ whenever $\E[P(F_n)] \to 0$ for some normalized random sequence $\{F_n\}_{n\ge 1}$ inside the second Wiener chaos then $F_n \to F_\infty$ in distribution.
\end{definition}


\begin{rem}{ \rm
The aim of the remark is to fairly clarify the role of other polynomials $P_n \in \PP$ in the characterization of the normal product distribution in the sense of Definition \ref{def:charactrize}. We have already shown that when $F$ is a normalized element in the second Wiener chaos and that $\E[P_6 (F)] =0$ then $F$ must distributed as the normal product distribution.  Concerning the polynomial $P_8$ in the characterization of the normal product distribution, consider the random element $F$ in the second Wiener chaos of the form
\begin{equation}\label{eq:F8}
          F =  - \frac 1 {\sqrt{3}}  (N_{-1}^2 -1) + \frac1{ \sqrt{ 12} }  (N_{1}^2 -1) +  \frac 1{\sqrt{ 12} }  (N_{2}^2 -1) =
         \frac{ N_{1}^2  + N_{2}^2 -2 N_{-1}^2     } {\sqrt{ 12} }.          
\end{equation}
We found the random element $F$ by using a random search algorithm. Note that $\E[F^2]=1$, and some straightforward computations yield that $\E[P_4(F)] = \E[P_8(F)]=0$, and furthermore $\text{law}(F) \neq \text{law}(F_\infty)$ where $F_\infty \sim N_1 \times N_2$. For the random variable $F$ in $(\ref{eq:F8})$ we have that   
$$F \stackrel{\text{law}}{=}\frac 1 { \sqrt{3} }\bigl(  G_1 G_2 + G_3 G_4 ) $$ where $(G_1,G_2,G_3,G_4)$ is  Gaussian vector with zero mean and  covariance matrix 
\begin{align*}
(C_{ij})_{1 \le i, j \le 4} = \left(\begin{matrix} 1 & 0& 1&  -1 \\
                      0& 1  &-1 & 1 \\
                      1 &-1&  1&  0 \\
                     -1 & 1 & 0 &  1
      \end{matrix}
\right)
\end{align*}
with the {\it Hankel matrix} form $\E( G_i G_j)=C(i+j)$ 
with $C(2)=C(8)=1, C(3)=C(7)=0, C(4)=C(6)=1,C(5)=-1$,
is the normalized sum of two dependent copies of  $N_1\times N_2$ random variables. One possibility for finding one root $\E[P_{2n}(F)]=0$ other than normal product distribution for higher values of $n$ is to consider polynomials  (see also item (b) of Remark \ref{rem:P6P8})
 $$Q_{n}(t):= \E[ P_{2n}(F_t)]= \E[P_{2n}(\sqrt{t} F_\infty + \sqrt{1-t} G_\infty)], \quad t \in [0,1]$$ where $G_\infty$ is an independent copy of $F_\infty$. It is easy to see that $Q_n$ is a polynomial in $t$ so that $\text{deg}(Q_n) = n$ when $n$ is even, and $\text{deg}(Q_n) = n-1$ when $n$ is odd, and hence $\text{deg}(Q_n)$ is always even.  See Figure \ref{fig:P10} for the graphs of polynomials $Q_n$, for $n=5,6,7$ on the interval $[0,1]$.
 As it can be seen polynomials $Q_5, Q_6$ and $Q_7$ have at least one real root in the open interval $(0,1)$. Moreover, one can show for $n \ge 5$ that polynomials $Q_n$ are of the form 
\begin{equation*}
Q_n (t) = 
\begin{cases}
\sum_{k=1}^{n}  (-1)^k\mathfrak{b}(n,k) \,t^k, & n \, \text{ even },\\
\sum_{k=1}^{n-1} (-1)^{k-1} \mathfrak{b}(n,k) \, t^k, & n \, \text{ odd }
\end{cases}
\end{equation*}
where coefficients $\mathfrak{b}(n,k)$ are all positive (non-zero) real numbers for every $1 \le k \le n ( \text{ or } n-1)$ depending whether $n$ is even or odd. Hence, as a direct consequence the number of the sign changes in $Q_n (x+0) = Q_n (x) = \text{deg}(Q_n) -1$ which is always an odd number. Furthermore  $Q_n(0)=Q_n (1)=0$. This directs one to the possibly use of the {\it Budan--Fourier Theorem} on the positive real roots of polynomials, and we leave it for further investigation later on. Finally, note that since the distribution $F_\infty$ is symmetric, and polynomials $P_k \in \PP$ for $k\ge 3$ being odd contain only odd powers of $x$, so the odd degree polynomials $P_n$ cannot be used in characterization of $F_\infty$ in the above sense.  

\begin{figure}
 \centering 
 \includegraphics[scale=0.5]{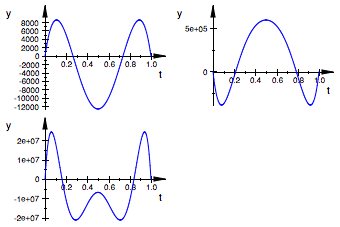}
  \caption{\small{ Polynomials $Q_5, Q_6, Q_7$}}
 \label{fig:P10}
\end{figure}

}
\end{rem}


\section{Appendix (A)}\label{sec:appendixa}

\begin{tiny}
\begin{align*}
P_{0}(x)&=1\\
P_{1}(x)&=x\\
P_2(x)&=x^2 -1\\
P_3(x)&=x^3-5  x\\
P_4(x)&=x^4-14x^2+5\\
P_5(x)&=x^5-30x^3+61 x\\
P_6(x)&=x^6-55x^4+331 x^2-61\\
P_7(x)&=x^7-91x^5+1211 x^3-1385 x\\
P_8(x)&=x^8-140 x^6+3486 x^4-12284 x^2+1385\\
P_9(x)&=x^9-204 x^7+8526 x^5-68060 x^3+50521 x\\
P_{10}(x)&=x^{10}-285x^8+18522 x^6-281210 x^4+663061 x^2-50521\\
P_{11}(x)&=x^{11}-385 x^9+36762x^7-948002 x^5+5162421x^3-2702765x\\
P_{12}(x)&=x^{12}-506x^{10}+67947x^8-2749340 x^6+28862471x^4-49164554 x^2+2702765\\
P_{13}(x)&=x^{13}-650x^{11}+118547 x^9-7097948 x^7+127838711x^5-510964090x^3+199360981x\\
P_{14}(x)&=x^{14}-819 x^{12}+197197x^{10}-16700255x^8+475638163 x^6-3706931865 x^4+4798037791x^2-199360981\\
P_{15}(x)&=x^{15}-1015 x^{13}+315133 x^{11}-36419955 x^9+1544454483 x^7-20829905733 x^5+64108947633 x^3-19391512144x\\
\end{align*}
\end{tiny}

\section{Appendix (B)}\label{sec:appendixb}

\begin{lem}\label{lem:appendixb}
Consider a random variable $X\in L^1(P)$ and  a random vector $Y\in \R^d$
with continuous density $p_Y(u)$. Then
\begin{align} \label{conditional:expectation}
  \E\bigl(  X \big\vert Y= u \bigr)= \frac{\E\bigl(   X \delta_u(Y) \bigr)  }
  { \E\bigl(  \delta_u( Y) \bigr) } \in L^1( \R^d, P_Y )
\end{align}
where $\delta_u(y)=\delta_0( y-u)$,$\delta_0$ is the Dirac delta function, and
$\E\bigl(  \delta_u( Y) \bigr)=p_Y(u)$  is the density of $Y$.
\end{lem}
\begin{proof} Note first that by definition of the Dirac delta
\begin{align*}
 \E\bigl(  \delta_u( Y) \bigr)= \int_{\R^d} \delta_u( y) p_Y(y) dy = p_Y(u).
\end{align*}

Also,
for  any bounded continuous test function $g: \R^d\to\R$ 
\begin{align}\nonumber\label{fubini:dirac}\E\bigl( X g(Y ) \bigr) &
=\int_{\Omega}\biggl( \int_{\R^d} \delta_0( Y(\omega) -u)     g(u) du\biggr) X(\omega) \P(d\omega) & \\ & 
= \int_{\R^d} \frac{ \E\bigl(   X \delta_0( Y-u ) \bigr) } { p_Y( u ) } g(u) p_Y(u) du.
 \end{align}
Using  Fubini theorem with the generalized function $\delta_0$ is justified as it follows:
for a sequence of mollifiers with compact support $\eta_n\to \delta_0$ (in distribution), for example
\begin{align*}
\eta_n(x)= n^d \prod_{i=1}^d( 1- n |x_i|)^+,
\end{align*}
we have
\begin{align*} &\int_{\R^d} \frac{ \E\bigl(  \eta_n( Y-u ) X\bigr) } { p_Y( u ) } g(u) p_Y(u) du 
= \int_{\Omega} \biggl( \int_{\R^d} \eta_n( Y(\omega) -u)     g(u) du\biggr)  X(\omega)\P(d\omega) & \\ & \longrightarrow 
 \int_{\Omega}\biggl( \int_{\R^d} \delta_0( Y(\omega) -u)     g(u) du\biggr) X(\omega) \P(d\omega)
 &\end{align*}
 since the measure $P(d\omega)\eta_n( Y(\omega)-u)du $ converges  in distribution towards $P(d\omega) \delta_{Y(\omega)}(u) du$.
Note also that
 the sequence of functions
 \begin{align*}
    \rho_n(u)= p_Y(u)^{-1}\E\biggl( \eta_n( Y(\omega) -u) X\biggr)
 \end{align*}
 is bounded in $L^1(\R^d,P_Y)$, and since the unit ball of  $L^1(\R^d,P_Y)$ is weakly compact,
 the sequence $\rho_n(u)$ 
converges weakly in $L^1(\R^d,P_Y)$ towards \eqref{conditional:expectation} which satisfies  \eqref{fubini:dirac}.
 The results extends to all bounded measurable $g$ by the standard monotone class argument.
\end{proof}

\end{document}